\documentclass[9pt,twocolumn,twoside,lineno]{pnas-new}

\RequirePackage{totcount,xpatch}

\usepackage{amsthm}
\usepackage{amsmath}

\usepackage{listings}
\usepackage{xcolor}

\definecolor{codegreen}{rgb}{0,0.6,0}
\definecolor{codegray}{rgb}{0.5,0.5,0.5}
\definecolor{codepurple}{rgb}{0.58,0,0.82}
\definecolor{backcolour}{rgb}{0.95,0.95,0.92}

\lstdefinestyle{mystyle}{
    backgroundcolor=\color{backcolour},   
    commentstyle=\color{codegreen},
    keywordstyle=\color{magenta},
    numberstyle=\tiny\color{codegray},
    stringstyle=\color{codepurple},
    basicstyle=\ttfamily\footnotesize,
    breakatwhitespace=false,         
    breaklines=true,                 
    captionpos=b,                    
    keepspaces=true,                 
    numbers=left,                    
    numbersep=5pt,                  
    showspaces=false,                
    showstringspaces=false,
    showtabs=false,                  
    tabsize=2
}

\usepackage{graphicx}
\graphicspath{{Images/}}

\DeclareMathOperator{\img}{im}
\DeclareMathOperator{\id}{id}
\DeclareMathOperator{\CRVR}{CRVR}

\newtheorem{theorem}{Theorem}
\newtheorem{proposition}[theorem]{Proposition}

\templatetype{pnasresearcharticle} 

\title{Persistent homology method to detect block structures in weighted networks}

\author[a]{Wooseok Jung}

\affil[a]{Mathematical Institute, University of Oxford, Radcliffe Observatory Quarter, Woodstock Road, Oxford OX2 6GG, UK.}

\significancestatement{
Topological data analysis (TDA) is a relatively young method to study complex datasets by detecting their overall shape. Specifically, it detects structures such as connected components, cliques, and even higher-dimensional boundaries in the datasets. Many network science works emphasize network topology, but applying topological data analysis into network simulations is an entirely different matter. First, it is not common to use topological data analysis in community detection tasks, and second, it can give information relating to the 'shape' of a network dataset, but not 'demographics'. We propose TDA can classify networks with distinct block structure by computing their topological barcodes, which sufficiently contain information relating to local communities and their interactions.}


\begin{abstract}
Unravelling the block structure of a network is critical for studying macroscopic features and community-level dynamics. The weighted stochastic block model (WSBM), a variation of the traditional stochastic block model, is designed for weighted networks, but it is not always optimal. We introduce a novel topological method to study the block structure of weighted networks by comparing their persistence diagrams. We found persistence diagrams of networks with different block structures show distinct features, sufficient to distinguish. Moreover, the overall characteristics are preserved even with more stochastic examples or modified hyperparameters. Finally, when random graphs whose latent block structure is unknown are tested, results from persistence diagram analysis are consistent with their weighted stochastic block model. Although this topological method cannot completely replace the original WSBM method for some reasons, it is worth to be investigated further.
\end{abstract}

\begin{document}

\maketitle
\thispagestyle{firststyle}
\ifthenelse{\boolean{shortarticle}}{\ifthenelse{\boolean{singlecolumn}}{\abscontentformatted}{\abscontent}}{}

\dropcap{N}etwork models interpret interactions in complex systems, from biological to social networks and from cellular scale to interactions among countries. But the complexity of real-world networks demands the reduction into low-scale through community detection.  We can study each community as an individual subnetwork and interactions among communities to point out the macroscopic features. Network science always goes with studies of communities, such as epidemic spreading \cite{PhysRevE.101.032309}, functional brain networks \cite{10.1371/journal.pcbi.1000381}, social network sites \cite{BediPunam2016Cdis}, and trading \cite{ReichardtJ.2007Rmfc}.

The stochastic block model (SBM) \cite{HOLLAND1983109} is a statistical inference method to obtain the community structure of a network. The SBM allocates each vertex $v_{i}$ to a block of vertices $C_{r}$ while generating the adjacency matrix of the blocks $A^{Block}_{ij}$ from a probability distribution. Then, it can recover the network whose nodes are the blocks and edges are the block relations.

However, SBM is limited to unweighted networks, but many real-world networks live in the weighted regime. Those weights are necessary to capture the detailed and accurate block structure of a network \cite{thomas2011valued}. The weighted stochastic block model (WSBM) \cite{aicher2013adapting, 10.1093/comnet/cnu026} learns the latent structure from both connectivity and weights and utilizes a variational Bayesian approach to optimize the model. Already WSBM is applied for the human connectome \cite{FaskowitzJoshua2018WSBM}, for example. 

On behalf of the WSBM, The goal of this project is to introduce a topological method which captures the latent block structure of weighted networks to some degree. Networks have various topological properties containing sufficient information for studying low-level interactions \cite{AktasMehmetE.2019Phon}. Although topological approaches cannot suggest exact, sharply-measured properties, we focus on their power to detect generalized configurations tenacious to small perturbations.

We first introduce background knowledge of topological data analysis. We also discuss the fundamentals of the weighted stochastic block model and its considerable limitations. In the next section, we introduce the cropped reciprocal Vietoris-Rips filtration (CRVR), which is a modification of the inverse Vietoris-Rips filtration \cite{10.1371/journal.pone.0066506}. The CRVR filtration fits more naturally in community detection tasks of weighted networks. We study how a block structure determines topological features via investigating assortative, disassortative, core-periphery, and ordered networks. Last, we compare results topological method and the WSBM from simulations through a random network. 

\section*{Backgrounds}
An undirected finite network $\mathcal{G} = (\mathcal{V}, \mathcal{E})$ is weighted if there exists a function $w: \mathcal{E} \longrightarrow \mathbb{R}_{\geq 0}$ assigning each edge to a non-negative weight. Therefore, the weight matrix $W$ of $\mathcal{G}$ can be defined as $W_{ij} = 0$ if vertices $v_{i}, v_{j}$ are disconnected and $W_{ij} = w_{ij}$ if two vertices are connected with the edge weight $w_{ij}$. 

Measures of structural properties can be defined similarly to the unweighted version. For example, the clustering coefficient of a weighted network with size $N$ is defined as follows \cite{Barrat3747}:
\begin{equation}
    C_{i} = \frac{1}{s_{i}(k_{i}-1)}\sum_{j,h}\frac{w_{ij} + w_{jh}}{2}a_{ij}a_{ih}a_{jh}
\end{equation}
where $a_{ij}$ is the $ij$-element of the adjacency matrix and $s_{i} = \displaystyle\sum_{j=1}^{N}a_{ij}w_{ij}$ is the strength of the node $i$. 

The weighted environment of a network facilitates the filtration of the topological complexes. The most fundamental configuration is the simplicial complex, which is defined by a pair $(V, K)$ with following rules: 
\begin{enumerate}
    \item For all $v \in V$, $\{v\} \in K$.
    \item For all $\tau \in K$, every subset of $\tau$ is in $K$.
\end{enumerate}
So, $\Sigma$ is a collection of subsets of $V$, and denote $K$ as a simplicial complex over $V$ rather than the pair notation for simplicity. Elements of $K$ are called simplices, and $\sigma$ is a face of $\tau$, denoting $\sigma \leq \tau$, if $\sigma \subset \tau \in K$. A subset $L$ of $K$ is a subcomplex if $L$ itself forms a simplicial complex over a subset of $V$.

A filtration of a simplicial complex $K$ is a nested sequence of subcomplexes of $K$ such that:
\begin{equation}
    F_{0}K \subset F_{1}K \subset F_{2}K ... \subset F_{n}K = K.
\end{equation}
On a finite metric space $(M,d)$, the Vietoris-Rips complex $VR_{\epsilon}(M)$ is a simplical complex where $\{v_{0}, v_{1} \dots , v_{k}\} \subset M$ is a simplex in $VR_{\epsilon}(M)$ if and only if $\max_{i,j}d(v_{i},v_{j}) \leq \epsilon$.  Clearly $VR_{\delta}(M)$ is a subcomplex of $VR_{\epsilon}(M)$ for any $0 \leq \delta \leq \epsilon$. Therefore, we can define the filtration over a continuous index. 

For each $k \geq 0$, the $k$-th chain group of a simplicial complex $K$ is the vector space $C_{k}(K)$ over $\mathbb{F}$ spanned by the $k$-simplicies of $K$. Then, the corresponding boundary map $\partial_{k}^{K}:$ $C_{k}(K) \longrightarrow C_{k-1}(K)$ sending each $k$-simplex to its algebraic boundary. Therefore, we can construct a sequence of chain groups connected by boundary maps, such as: 
\begin{equation}
    \cdots C_{k}(K) \xrightarrow{\partial_{k}^{K}} C_{k-1}(K) \xrightarrow{\partial_{k-1}^{K}} \cdots \xrightarrow{\partial_{1}^{K}} C_{0}(K) \xrightarrow{} 0.
\end{equation}
The collection $(C_{*}(K), \partial_{*}^{K})$ is a chain complex if $\partial_{k}^{K} \circ \partial_{k+1}^{K} = 0$ for all $k \geq 0$. The $k-$th homology of this chain complex is, roughly speaking, a vector space spanned by $k-$dimensional cycles which are not boundaries of $k+1$-dimensional manifolds. To be more precise, 
\begin{equation}
    H_{k}(C_{*}, \partial_{*}) = \ker{\partial_{k}} / \img{\partial_{k+1}}.
\end{equation}

Fixing the dimension $k$, the $k-$th homology groups of a given filtration induces another sequence of vector spaces with associated linear map $a_{i \rightarrow j}:$ $H_{k}(F_{i}K) \xrightarrow{} H_{k}(F_{j}K)$ induced from the simplicial inclusion map $\iota_{i \rightarrow j}: F_{i}K \hookrightarrow  F_{j}K$. This sequence $(V_{*}, a_{*}) = (H_{k}(F_{*}K), a_{*})$ is called the persistence module, and finally, the persistent homology group of it is given by
\begin{equation}
    PH_{i \rightarrow j}(V_{*}, a_{*}) = \img{a_{i \rightarrow j}}.
\end{equation}

Note that the persistence homology gives information about where a non-boundary loop arises and when it dies. For example, if a $k$- dimensional loop is born at $t = t_{1}$, then it becomes an element of a basis of $H_{k}(F_{t}(K))$. If it disappears at $t = t_{2}$, then the map $a_{t_{1} \rightarrow t_{2}}$ maps the loop to zero. But for all $t_{3} < t_{2}$, $a_{t_{1} \rightarrow t_{3}}$ maps the loop to a non-trivial element. See more concrete definitions of relevant topological terminologies in \cite{JonssonJakob2005SCoG}.
\section*{Barcodes}
The barcode of a persistence module contains information about when the loops are born and die out. To begin with, for each pair of $0 \leq i \leq j$, the interval module $(I_{*}^{i,j}, c_{*}^{i,j})$ over a field $\mathbb{F}$ is given by
\begin{equation*}
  I_{k}^{i,j} =   
  \begin{cases}
  \mathbb{F} \hspace{0.3cm} \text{if} \hspace{0.1cm} i \leq k \leq j \\
  0 \hspace{0.3cm} \text{otherwise}
  \end{cases}
  \text{and}
  \hspace{0.25cm}
  c_{k}^{i,j} = 
  \begin{cases}
  \id_{\mathbb{F}} \hspace{0.3cm} \text{if} \hspace{0.1cm} i \leq k \leq j \\
  0 \hspace{0.3cm} \text {otherwise}
  \end{cases}.
\end{equation*}
So, we may regard an interval module as a bar of length $j - i$. Then, the structure theorem of persistence modules \cite{botnan2019decomposition} allows to uniquely decompose any persistence module to a direct sum. For any persistence module $(V_{*}, a_{*})$, there exists a set of intervals $B(V_{*}, a_{*})$ such that 
\begin{equation}
    (V_{*}, a_{*}) \cong \bigoplus_{[i,j] \in B(V_{*}, a_{*})}(I_{*}^{i,j}, c_{*}^{i,j})^{\mu(i,j)},
\end{equation}
where $\mu(i,j)$ is the multiplicity of the interval $[i,j]$. Hence, each interval module denotes a lifespan of a loop which arises at $i$ and disappears at $j$. 

Figure \ref{figs1} shows an example of a filtration of a simplicial complex $K$ and its corresponding barcode. Remark that the 0-dimensional homology group is equivalent to space spanned by one's connected components. At $t = 0$, there are three connected components, where each of them is a vertex, namely $v_{0}, v_{1}, v_{2}$. So, $H_{0}(F_{0}K) = \mathbb{F}^{3}$. At $ t = 1 $, one of the connected components is 'absorbed' to another, hence the bottom red bar ends. At $t = 2$, all of the components are now connected, however, there is a loop consisting of the edges of a triangle, and that's why $H_{1}$ starts to be non-trivial. At $t = 3$, the interior of the triangle is filled, so the loop is now disappeared. The single connected component remains, but now $H_{1}$ is dead. From those homology groups, the persistence homology $PH_{0 \rightarrow 1}$, for instance, is $F^{2}$ since it maps $\mathbb{F}^{3}$ to $\mathbb{F}^{2}$ for which $v_{0}$ and $v_{1}$ are mapped to same basis element and $v_{2}$ to the other. 

\subsection*{Computing Barcodes}
Many open-source libraries allowing us to compute persistent homology and find barcodes. For example, Ripser \cite{ctralie2018ripser} and Gudhi \cite{10.1007/978-3-662-44199-2_28} are developed to be used in Python environment. Ripser sticks on computing persistence homology of a Vietories-Rips filtration, but Gudhi is more flexible: accepting customized filtrations. 

\section*{Weighted Stochastic Block Model}

Let $A_{ij}$ be the adjacency matrix of an unweighted network $\mathcal{G}$. Suppose $\mathcal{G}$ contains $K$ latent blocks, and assign each vertex $v_{i}$ to the block index $z_{i} \in {1, \dots K}$. Define another matrix $\theta$ with $0 < \theta_{z_{m}z_{n}} < 1 \forall m,n$ where $\theta_{z_{m}z_{n}}$ represents the probability of connection exists between $z_{m}$ and $z_{n}$. Goal of the stochastic block model is to find the optimal $\theta$ which is the maximum likelihood estimator of $\mathbb{P}(A | z, \theta)= \prod_{i,j}\theta_{z_{i}z_{j}}^{A_{ij}}(1-\theta_{ij})^{1-A_{ij}}$. 

The weighted stochastic block model \cite{10.1093/comnet/cnu026} has an additional step to draw a weight $A_{ij} \in \mathcal{X}$ from the exponential family $(T, \eta)$ parametrized by $\theta$. Indeed, if the adjacency matrix is replaced by the weight matrix, then $A_{ij} = 0$ implies either $i$ and $j$ are not connected or connected but the edge weight is $0$. Therefore, the governing formula of the WSBM samples $A_{ij}$ from distributions from two different exponential families $(T_{E}, \eta_{E})$ and $(T_{W}, \eta_{W})$, the exponential family of edge-interaction distributions and weight distributions respectively. The log-likelihood of the probability of $A$ given $z$ and $\theta$ is:
\begin{equation}
\begin{split}
 \log\mathbb{P}(A | z, \theta) &= \alpha\sum_{i,j \in E}T_{E}(A_{ij})\eta_{E}(\theta_{z_{i}z_{j}}) \\ &+ (1-\alpha)\sum_{i,j \in W}T_{W}(A_{ij})\eta_{W}(\theta_{z_{i}z_{j}})   
\end{split}
\label{eq6}
\end{equation}
where the tuning parameter $\alpha \in [0,1]$ determines the importance of the two families. Furthermore, the degree correction method \cite{PhysRevE.83.016107} can be applied to make the network have heavy-tailed distribution. 

\subsection*{Optimization}

A variational Bayesian method \cite{VarBayes,PhysRevLett.100.258701} is applied to find the optimal MLE of the equation \ref{eq6}. The goal is to find the optimal posterior $\pi{*}(z,\theta) = \mathbb{P}(z, \theta | A) \propto \mathbb{P}(A | z, \theta)\pi(z, \theta)$ by Bayes theorem. As other Bayesian methods, it is impossible to find an analytic solution in general, so we need to approximate $\pi^{*}$ to some variational distribution $q(z, \theta)$, assuming all the latent parameters are independent:  $q(z, \theta) = q_{z}(z)q_{\theta}(\theta)$. 

Let the joint distribution of the data $A$, latent variable $z$, and the model parameter $\theta$ be $p(A,z,\theta)$. The variational free energy of this joint distribution is 
\begin{equation}
    \mathcal{L}(q, \theta) = \mathbb{E}_{z \sim q}[\log p(A,z,\theta) - \log q(z,\theta)].
\end{equation}
Then, after simple algebra, we get 
\begin{equation}
    \mathcal{L}(q, \theta) = \log p(A) - \mathbb{D}_{KL}(q(z,\theta) || p(z,\theta | A))
\end{equation}
where $\mathbb{D}_{KL}$ denotes KL-divergence. Since KL-divergence is always non-negative, $\mathcal{L}(q, \theta)$ forms a lower bound of the log-evidence $\log p(A)$. Therefore, variational Bayesian approach targets to minimize the KL-divergence  i.e. maximize the lower bound $\mathcal{L}(q,\theta)$. 

There are several methods to find variational approximations \cite{10.1093/comnet/cnu026}, or we can adapt more well-known algorithms like variational expectation maximization or coordinate ascent. 

\subsection*{Shortfalls} As other Bayesian methods, WBSM is also sensitive to the choice of prior distribution $\pi(z, \theta)$. A subjective choice of the prior is possible to lead to imprecise results.  Moreover, in the optimization process, $\mathcal{L}(q, \theta)$ can be approximated to a local optimum, not global. Finally, the variational inference is prone to producing highly biased estimates and hard to be generalized \cite{8588399}. 

\section*{Theoretical Results}

The weight of an edge can infer the distance between two vertices.  So, a large weight implies a strong, structurally close relationship. Viewing a weight matrix as a distance matrix, we can bring topological notions introduced previously. The topological method first accesses the pairwise distances of the nodes. Then, it captures the shape of the network data showing a mutual relationship with the community structure.

\subsection*{CRVR Filtration} The Vietoris-Rips filtration of a weighted network introduced in \cite{10.1371/journal.pone.0066506} is defined as a descending sequence of subcomplexes of which $ \{v_{0}, \dots, v_{n}\} \in F_{\epsilon}\mathcal{G}$ if and only if its pairwise weight is greater than $\epsilon$. Hence, $F_{\epsilon_{1}}\mathcal{G} \geq F_{\epsilon_{2}}\mathcal{G}$ for all $\epsilon_{1} \geq \epsilon_{2}$. 

However, constructing such descending sequence is seemingly unnatural. So, we define a developed method called the Cropped Reciprocal Vietoris-Rips filtration (CRVR) starting from defining the distance matrix $D$ with the following rule: 
\begin{equation}
    D_{ij} = 
     \begin{cases}
    \displaystyle 0 \hspace{0.3cm} \text{if} \hspace{0.1cm} i = j \\
    \displaystyle 1/w_{ij}\hspace{0.3cm} \text{if} \hspace{0.1cm} i \neq j \hspace{0.1cm} and \hspace{0.1cm} 0 < \zeta < w_{ij} \\
    \displaystyle 1/\zeta \hspace{0.3cm} \text{otherwise}
  \end{cases}
\end{equation}
with a pre-defined cropping parameter $\zeta$. Here the weights must be non-negative larger $w_{ij}$ implies smaller distance $D_{ij}$ between $v_{i}$ and $v_{j}$. Unlike the WSBM, the target network must be simple and fully-connected but make an edge $(v_{i}, v_{j})$ have arbitrarily small weight (including zero weight) to represent that $v_{i}$ and $v_{j}$ are disconnected. $\zeta$ determines the threshold of disconnectedness and fixes the distances of such pairs of vertices be $1/\zeta$. 

Then, the CRVR filtration $\CRVR_{*}(\mathcal{G})$ is a collection of simplicial complexes constructed as following: 
\begin{enumerate}
    \item $\{v_{i}\} \in \CRVR_{0}{(\mathcal{G})}$ for all $v_{i} \in \mathcal{V}$.
    \item $\{v_{0}, \dots v_{n}\} \in \CRVR_{k}{(\mathcal{G})}$ if and only if $\max_{i,j}D_{ij} \leq k$.
\end{enumerate}
It is a well-defined filtration, and for each $\CRVR_{n}$, we can find its $k-$th homology $H_{k}$ (dropped other notations for simplicity) and can define the corresponding persistence module. Then we can find its persistence homology and draw barcodes. 

\subsection*{Hypothesis} We expect in a weighted network, vertices with small pairwise distance form a community. Although this might clash with conventional notions of clustering in weighted networks, we assume this approach works in examples where weights are more like distance such as social networks defining the weight as mutual intimacy \cite{LI2020122894}.

\section*{Numerical Experiments}

Figure \ref{figs2} displays four different clustering types of networks with four 'groups'. Here we use the term 'group' instead of community to ensure that those groups of vertices do not always form a community in some algorithms. In the example networks, weights of edges connecting vertices from the same block are uniformly sampled from the interval $[1, 10]$, and those connecting vertices of different communities are from $[0,1]$. So, as our assumption to construct the CRVR filtration, the adjacency matrix is filled by $1$ except the diagonal, but the weight matrix is not. 

An assortative network has the most likely structure of a network with four clusters: it has highly connected within groups with weak connections among the groups. Green blocks in the diagram of the weight matrix represent edges with the weight between 1 and 10, and dark blocks show edge weights between 0 and 1. A disassortative network is the opposite of the assortative example: only edges between groups have large weights. The core-periphery structure has one central community with high weights inside and the other communities are strongly connected to the weights, but showing both weak inter-and intra- connections in the other three. Finally, the ordered network has a snake-like structure. It connects adjacent vertex groups with strong weights. These examples showed in \cite{10.1093/comnet/cnu026} divides the cases by connections or non-connection, but here we formulate 'no connections' by small weights. 

\begin{table}[tbhp]
\centering
\caption{Number of partitions and modularity found by the Louvain method}
\begin{tabular}{lrrr}
Type & Number of Partitions & Modularity \\
\midrule
1. Assortative & 4 & 0.5893 \\
2. Disassortative & 4 & 0.0233 \\
3. Core-Periphery & 4 & 0.0454 \\
4. Ordered & 2 & 0.2750 \\
\bottomrule
\end{tabular}
\label{tb1}
\end{table}
\begin{figure}[tbhp]
    \centering
    \includegraphics[width = 0.9\linewidth]{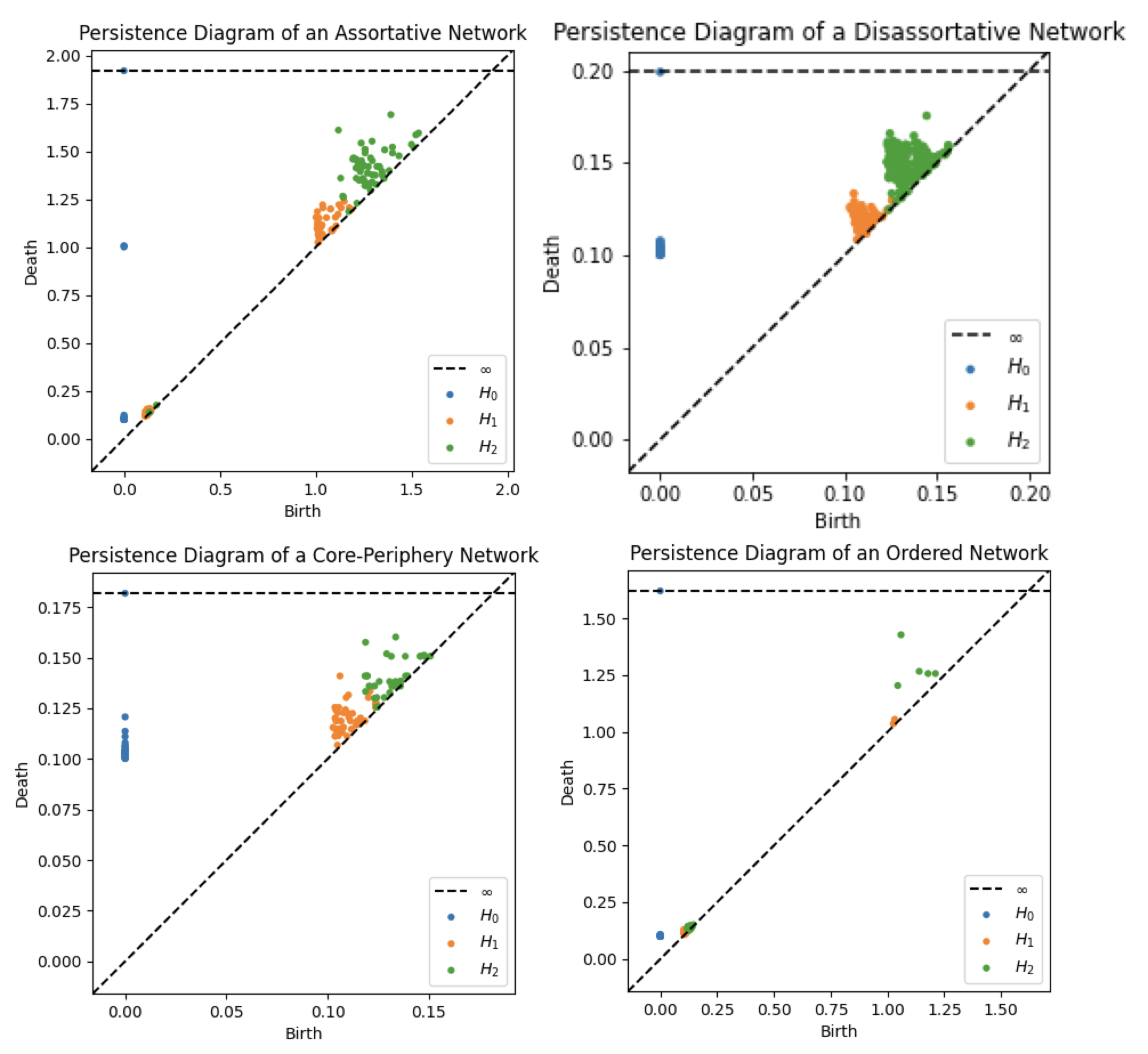}
    \caption{Persistence diagrams of a networks with different block structures: assortative, disassortative, core-periphery, and ordered. Blue, red, and green circles denotes 0, 1, and 2-dimensional barcodes respectively.}
    \label{fig1}
\end{figure}


Table \ref{tb1} shows the number of partitions of the four example networks and their modularities found by the Louvain method \cite{Blondel_2008}. The result of the assortative example is very understanding since it has a clear community structure. However, disassortative and core-periphery examples have very low modularity even though there are four partitions. Indeed, this is the result of the only four networks in which we draw their persistence diagrams, not the average, so while repeating the simulations, several partitions of disassortative and core-periphery networks are not always four: sometimes five or six, even seven. The ordered example showed higher modularity than the previous two, but the number of partitions is two.

Persistence diagrams are graphs to display persistence barcodes. We use the Ripser \cite{ctralie2018ripser} library to plot them. Birth is the index number of the filtration when a barcode starts, and death denotes when it dies. So, simply a persistence diagram depicts the boundary points of barcodes, and they always lie above $y = x$, the dashed diagonal line in a diagram. In the assortative network persistence diagram of Figure \ref{fig1}, we can see the $0-$dimensional homology groups and the connected components are all born at $t = 0$. Some of them die very fast, and some demise at $t = 1$ and then remains infinitely. The distances within a community are at most 0.5, so points within a group rapidly form a single connected component but then absorbed to others starting from $t=1$, and that's why some $0-$dimensional barcodes die at $t = 1$. Note that a single point in the diagram might denote multiple overlapped points. Also, there are $1-$ and $2-$dimensional points born at $t < 0.5$ and dies immediately, and clearly, they are the ones within a community. Moreover, there are even more orange and green points rising after $t = 1$, coming from edges between the vertex groups.

The persistence diagram (PD) of the disassortative example looks the same as Figure assortative PD, but beware that $x$ and $y$ values are different. There is already a single connected component at $t = 0.2$. Also, $H_{1}$ arises from $0.1$, the minimum edge weight. Here 1-dimensional loops are from vertices in different communities. After then, we can check 2-dimensional barcodes are born. 



In the PD of the core-periphery network, the overall configuration of the points looks the same to disassortative, but note that now the orange and green regions overlap. Here all the 1-dimensional loops come from the core group, because the distance between vertices within the core is small, and forming a loop connecting different communities requires a connection between two groups except the core, but such edges have small 'weight.' Similarly, unlike the disassortative network where a 2-dimensional loop whose vertices from different groups. 2-dimensional loops in core-periphery are like a cone. The 'base' is in the core group, and the apex is in another group. But, since edges within a group of a disassortative network have small weight, such cone-like 2-dimensional hollow simplicial complex (i.e. tetrahedron with empty interior) is not relevant.

Finally, an ordered network has a significantly small number of barcodes. The network forms a huge connected component very rapidly since every group is connected sequentially. Like the assortative network, $H_{1}$ and $H_{2}$ barcodes are born around $t = 0.1$, but such loops rarely arise at a higher filtration index since it doesn't form a loop of blocks. Consider a loop of length four $(v_{1}, v_{2}, v_{3}, v_{4}, v_{1})$ whose vertices are from different groups. In an assortative network, such loop requires $t \geq 1$ by the definition of weights of edges between groups. However, in an ordered network, even though the edge $(v_{4},v_{1})$ has a weight $\leq 1$, there could be a path $(v_{4}, v_{3}', v_{2}', v_{1})$ whose each edge weight is greater than one. Concatenating this path to $(v_{1}, v_{2}, v_{3}, v_{4})$ forms a loop which is homotopic to the first one. Hence, such $H_{1}$ and $H_{2}$ loops are seldom in an ordered network but not in the assortative. 

\begin{figure}
    \centering
    \includegraphics[width = .9\linewidth]{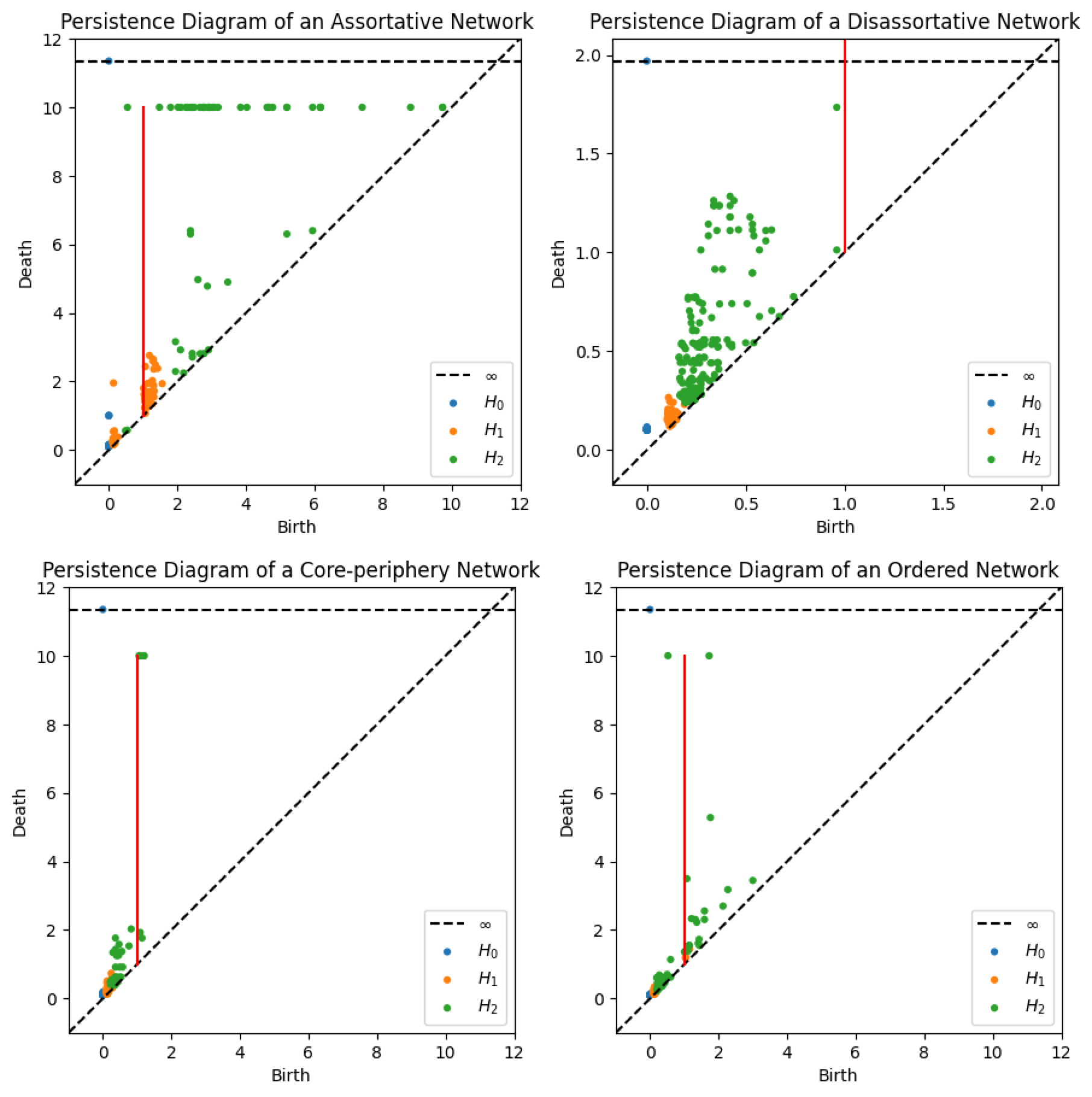}
    \caption{Persistence diagrams of four different clustering types with $p = 0.7$ and $q = 0.3$. The red vertical line in each plot denotes $x = 1$.}
    \label{fig5}
\end{figure}
Not every network has strict weight distributions within or between groups. Suppose vertices within a group are connected by probability $p$ and weights are drawn from $U[1,10]$. Vertices between groups are connected by probability $q$ and weights are drawn from $U[0.1,1]$. 'Non-edges' are regarded as edges with a weight of 0.1. 

Figure \ref{fig5} shows the result. In assortative network, there exists a connected component dies at $t = 1$, and lots of $H_{2}$ barcodes dies at $t = 10$, the maximum distance, as well as some $H_{1}$ barcodes born after $t = 1$. The disassortative network shows all barcodes born at $t < 1$ as before, but we can see some survives longer than those of the disassortative network in Figure \ref{fig1} due to the existence edges with larger weights. In core-periphery, a few $H_{2}$ barcodes born after $t = 1$. The ordered network still shows a relatively small number of barcodes after $t > 1$ than the assortative network, and most of them die earlier. 

\begin{figure}
    \centering
    \includegraphics[width = .88\linewidth]{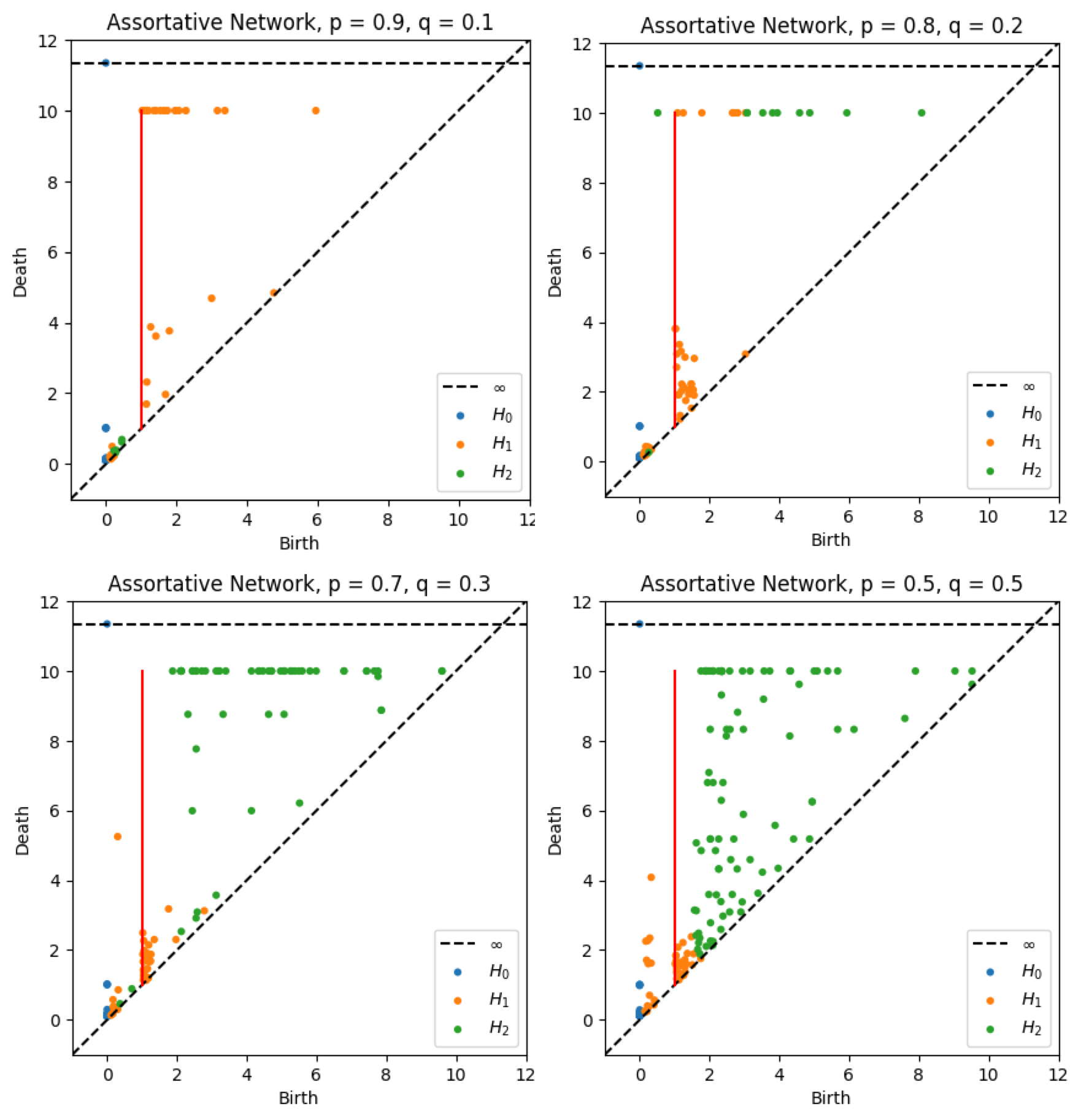}
    \caption{Persistence diagrams of assortative networks with $p$ and $q$ varied.}
    \label{fig6}
\end{figure}

Next, we test how $p$ and $q$ affect the barcodes with fixed block structure. Considering only the assortative block structure, Figure \ref{eq6} displays the number of $H_{2}$ barcodes at $t \geq 1$ increases with $q$ whereas number of $H_{1}$ barcodes increases at $t < 1$. Larger $p$ implies stronger connections within a group, suppressing  $H_{1}$ loops within a group but producing $H_{2}$ loops. Since a small $q$ weakens connections between groups, edges of a $H_{2}$ loop tend to have the minimum weights, making them demise at $t = 10$. When $q$ becomes larger, vertices in the same group likely form $H_{1}$ loops which survive longer, but much more $H_{2}$ loops of groups arise and die before $t = 10$, since more edges among groups prone to have weights greater than 0.1.  

\begin{figure}[tbhp]
    \centering
    \includegraphics[width = \linewidth]{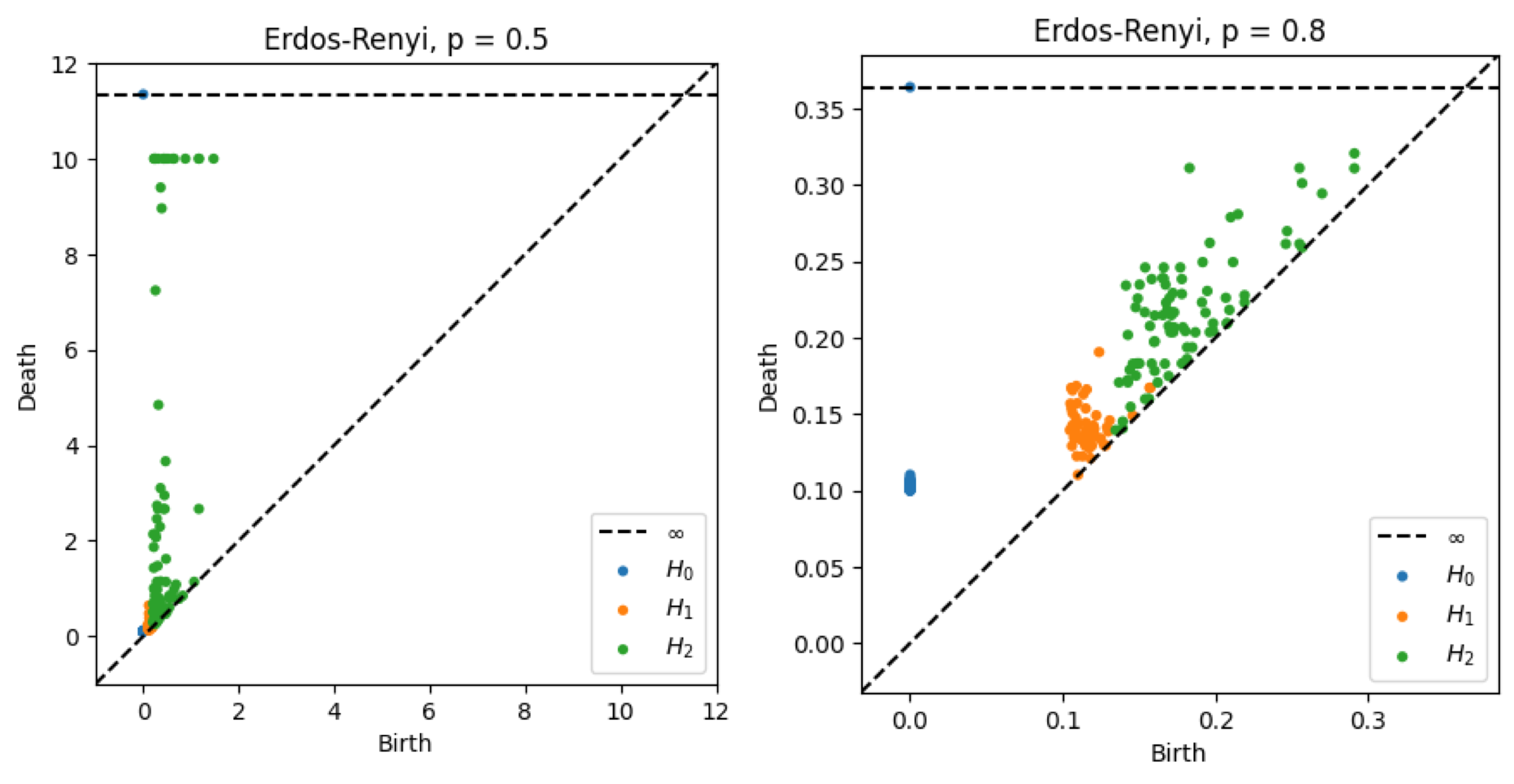}
    \caption{Persistence diagrams of Erdos-Renyi model with the connection probability $p = 0.5$ (Left) and $p = 0.8$ (Right). After generating graphs, weights are picked randomly from $[0.1, 10]$ and assigned to each edge}
    \label{fig7}
\end{figure}
\begin{figure}[tbhp]
    \centering
    \includegraphics[width = \linewidth]{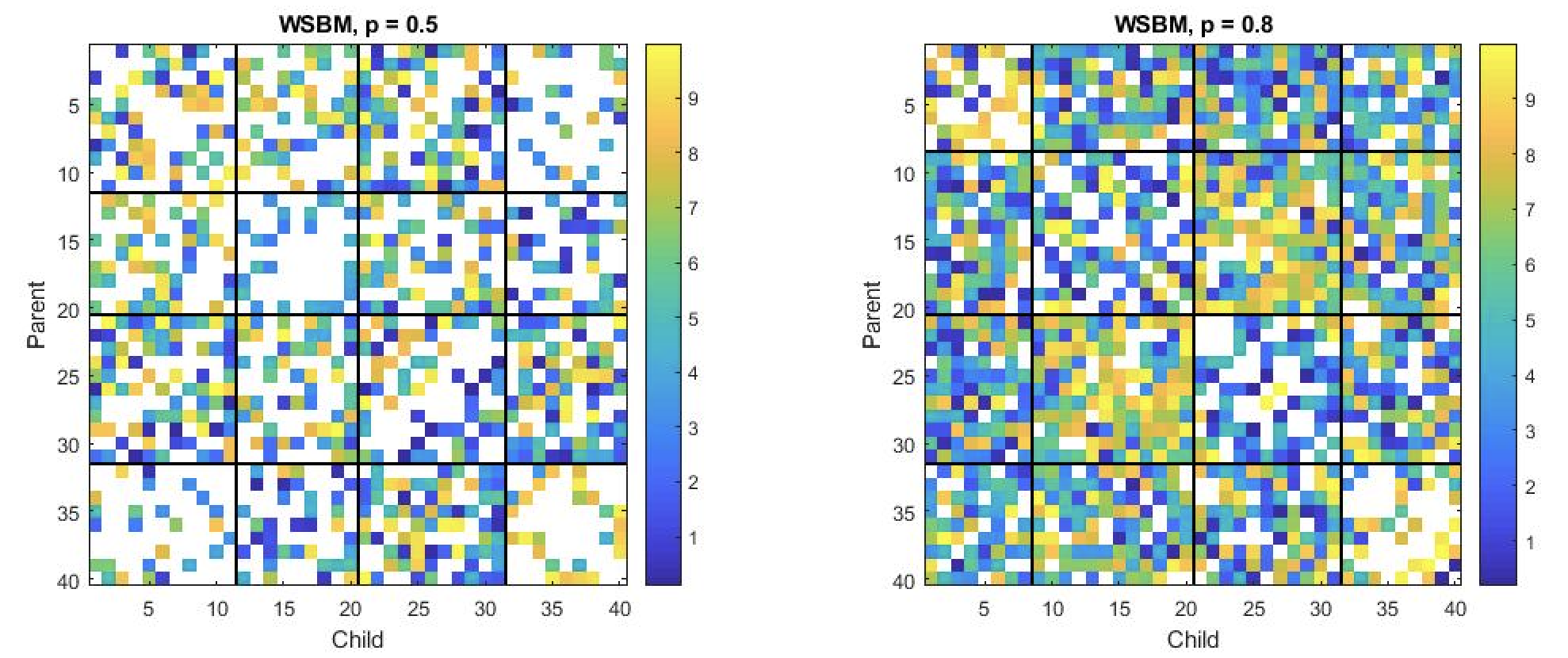}
    \caption{Weighted stochastic block model of weighted Erdos-Renyi graph from Figure \ref{fig7}. The algorithm automatically detects and allocate edges into the initialized number of blocks.}
    \label{fig8}
\end{figure}
We have simulated networks with known block structure until now and showed a significant difference among the structures. However, This naturally induces a question if persistence homology can be used for detecting the unknown block structure of networks. So, we generate Erdos-Renyi graphs, randomly assign edge weights, draw their persistence diagrams, and compare them with the WSBM of the graphs. In Figure \ref{fig7}, when $p = 0.5$, $H_{0}$ disappears immediately and all barcodes are born very fast; some $H_{2}$ barcodes stay for a long time. When $p$ increases to 0.8, all is the same except every barcode, including $H_{2}$, dies fast. Comparing these with Figure \ref{fig1} and \ref{fig5}, we can point out they are similar to core-periphery and disassortative example, respectively. Moreover, on the left, assuming the core-periphery structure, we can see there are more $H_{2}$ barcodes dying later than the core-periphery diagram of Figure \ref{fig5}. This result provides connections between non-core blocks stronger than the default core-periphery structure: disappearing earlier implies a shorter pairwise distance of a loop. 

Will this coincide with the results from the WSBM method? In both simulations, we assume the prior distribution of edge weight to Gaussian and that of edge existence to Bernoulli. In Figure \ref{fig8}, we have two weight heatmaps classified by the algorithm and divided into four blocks (bold black lines). At $p = 0.5$, WSBM is similar to the core-periphery structure in Figure \ref{figs2} with the core at the third block: the third row and column of the block are denser than others. At $p = 0.8$, the diagonal blocks are more sparse than the others, similar to the disassortative structure.  However, unlike the WSBM in Figure \ref{fig8}, persistence diagram analysis cannot provide how strong connections among blocks are. For instance, block 2 and 3 are densely connected at $p = 0.8$, but there is no way to figure it out in Figure \ref{fig7}.
\section*{Conclusions}
Portraiting persistence diagrams of networks and comparing them to the results from the WSBM confirm we can apply topological data analysis for capturing the latent block structure. Patterns of the persistence barcodes from the example structures with four vertex groups are significantly distinct. Some characteristics are perpetual even if the default network is perturbed by modifying the spectrum of weights in groups and pairwise connectivity. Furthermore, persistence diagram analysis of a network with unknown latent structure accurately classified the block type after validating them with the WSBM. 

However, there are some limitations of our topological method to replace the conventional WSBM. First, it does not provide sufficient information for community detection. Most importantly, the number of latent blocks is unknown. For example, assortative networks with 4 clusters and 8 clusters do not show a notable difference between their persistence diagram. Also, it does not reveal the address of each vertex. Although the persistent homology method can detect the general latent structure, we cannot determine the connections. Additionally, our construction of filtration does not embed the network data into a metric space. Vietoris-Rips filtration is defined only on a metric space, but our definition of the distance matrix $D$ in the CRVR filtration is not a metric since it does not satisfy the triangle inequality. Although it does not contradicts the definition of persistence homology, but it leaves some issues. See the Supporting Information section A. Finally, generalization to unweighted networks is difficult. We need to compute the length of the shortest paths between all pairs of vertices, but it is computationally expensive if the network is massive. Indeed, if the network is small enough, it could show better performance since the shortest path measure is precisely a metric.   

Even though there are many limitations, our method for uncovering the latent block structure is worth to be investigated more because of its stability in noisy data and capability to encapsulate the rough structure. Therefore, if a project aims to simply capture and classify the block structures of networks, then the topological method might be applicable. More complicated tasks with TDA are available. From the persistence diagrams, we can measure the distance between sets of barcodes by computing their persistence landscape \cite{BubenikP2015STDA}. Afterwards, we can vectorize those topological values and utilize them for the classification of networks, or unsupervised learning works to observe clustering of networks with various block structure. We hope our new approach ignites a new pathway in network science. 

\bibliography{main}

\newpage

\regtotcounter{section}
\regtotcounter{figure}
\regtotcounter{table}
\regtotcounter{NAT@ctr}  
\setcounter{figure}{0}
\setcounter{theorem}{0}
\renewcommand{\thetable}{S\arabic{table}}
\renewcommand{\thefigure}{S\arabic{figure}}

\section*{Supporting Information}
\subsection{Why metric condition is important}
Let $(M,d)$ be a finite metric space with the associated metric $d$. Assume we construct the Vietoris-Rips complex $VR_{\epsilon}(M)$ of this space. Suppose there exists a point $x \in M$ such that $\max_{y \in M}{d(x,y)} = \epsilon$. Then we have the following proposition: 
\begin{proposition} 
$H_{k}(VR_{\epsilon}(M)) = 0 $ for all $k > 0$.
\end{proposition}
\begin{proof}
Suppose we have a $k-$simplex $\{y_{1}, \dots, y_{k}\}$ constituting a $k-dimensional$ cycle. The $k-$simplex is a $k-$face of the simplex $\{y_{1}, \dots, y_{k}, x\}$ by the definition of $x$, this is a solid $k+1$-simplex whose geometric realization is contractible. This implies every component of a $k-$cycle retracts to a contractible module, hence the $k-$th homology is trivial. 
\end{proof}
This implies there is no need to go after the filtration index $\epsilon$ if such $x$ exists. In the $\CRVR_{\epsilon}{(M)}$, however, a loop of edge length $> \epsilon$ which is not contractible can exist. So, there is a possibility that the persistence diagram cannot investigate all cycles of interest in our framework. But such $x$ rarely exists at small $\epsilon$ in general.  
\subsection{Figures} Here is the list of supplementary figures.

\begin{figure}[htbp]
\centering
\includegraphics[width=.9\linewidth]{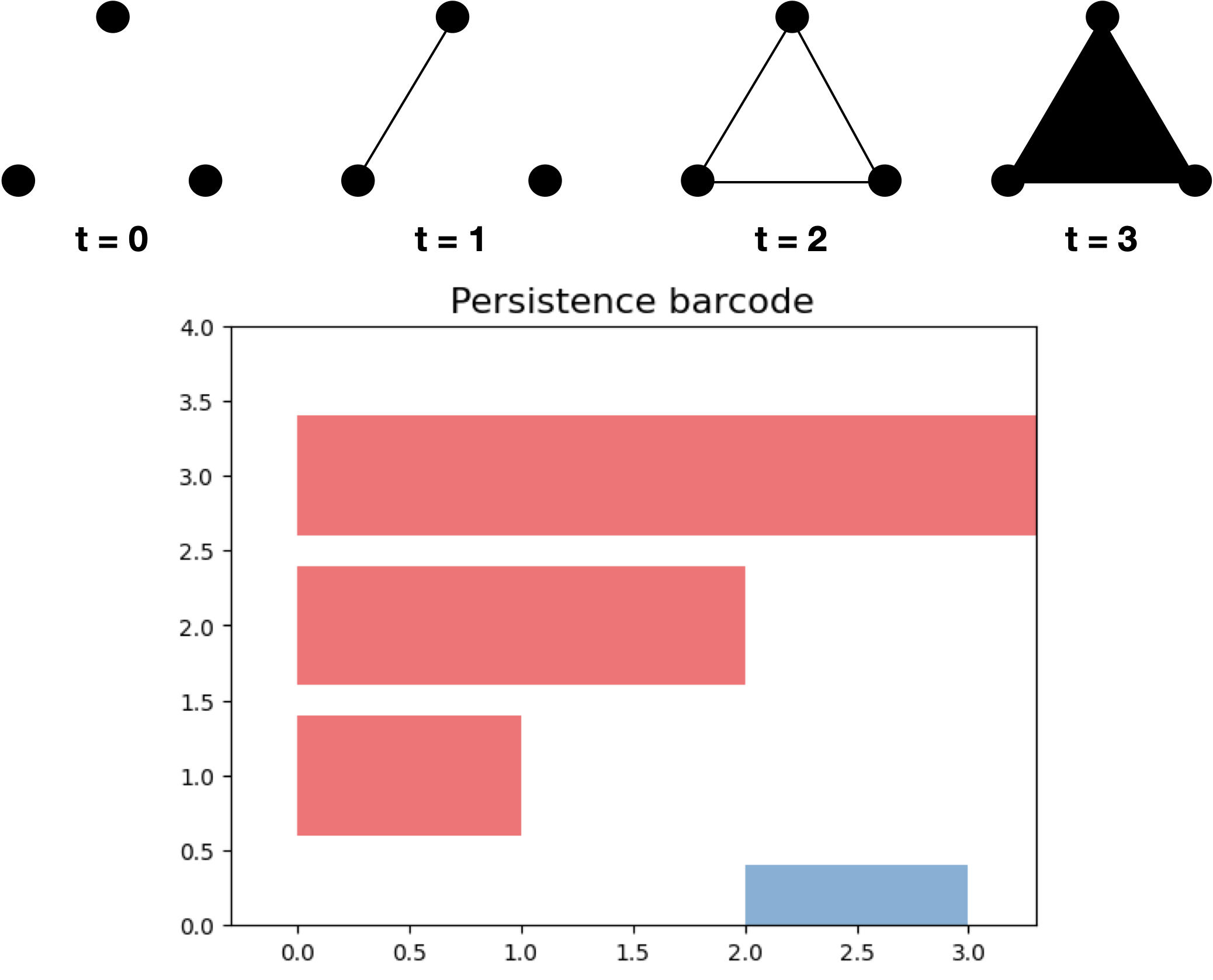}
\caption{An example of a filtration of the solid 2-simplex and its corresponding barcodes. The red bars denote the barcodes of $H_{0}$ and the blue one denotes $H_{1}$}
\label{figs1}
\end{figure}

\begin{figure}[htbp]
\centering
\includegraphics[width=.9\linewidth]{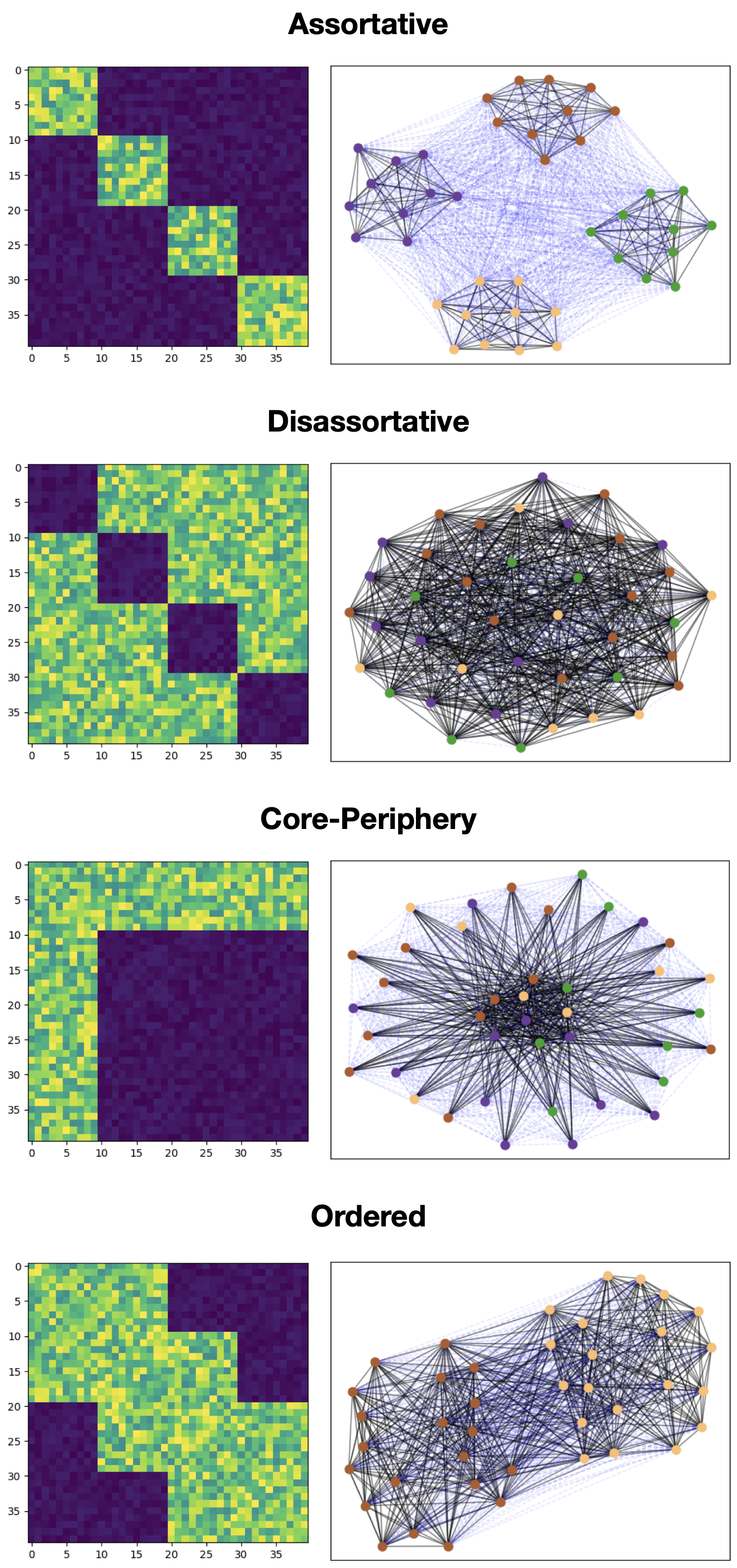}
\caption{Weight matrices (Left) and corresponding graphics of networks (Right) with four different block structures. All networks have 40 vertices, 10 vertices per each group. On the right column, vertices in the same community detected by the Louvain algorithm have the same color. The edges with weight $\geq 1$ are colored by black, and the blue, dashed and faint edges have weight $ < 1$}
\label{figs2}
\end{figure}

\end{document}